\documentclass[12pt,reqno]{amsart}
\usepackage[hmarginratio=1:1]{geometry}

\usepackage{amsmath}
\usepackage{amsfonts,mathtools}
\usepackage[all]{xy}
\usepackage{amssymb}
\usepackage{xcolor}
\usepackage{dsfont}
\usepackage{color}
\usepackage{cite}
\usepackage{graphicx}
\usepackage{braket}
\usepackage{slashed}
\usepackage{float}
\usepackage{braket}
\usepackage{color}
\usepackage{upgreek}
\usepackage{hyperref}
\usepackage{physics}
\usepackage{tikz}
\usepackage{tikz-cd}
\def\hall{\text{Hall}_{\text{mot}} (\rep)}
\def\mod{\text{Mod-}kQ}

\hypersetup{
    colorlinks,
    linkcolor={blue!50!red},
    citecolor={red!70!red},
    urlcolor={orange!80!black}
} 

\def\a{\alpha}

\def\vir{\text{vir}}

\def\dim{\text{dim}}

\def\Ext{\text{Ext}}
\def\Hom{\text{Hom}}
\def\Aut{\text{Aut}}
\def\End{\text{End}}
\def\ext{\epsilon_{A,B}}

\def\be{\begin{equation}}
\def\ee{\end{equation}}

\def\A{\mathbb{A}}

\def\ko{K_0(\text{Var}_k)}

\def\varki{\text{Var}_k}

\def\a{\mathbb{A}}
\def\P{\mathbb{P}}
\def\Hilb{\mathbb{\textbf{Hilb}}}

\def\L{\mathbb{L}}
\def\M{\mathcal{M}}
\def\ak{\mathbb{A}_k}
\def\md{\left[ \frac{R_d(Q)}{G_d} \right]}
\def\rep{\text{Rep}_k(Q)}
\def\la{\Lambda}

\theoremstyle{definition}
\newtheorem{point}{Point}

\newtheorem{remark}{Remark}
\newtheorem{theorem}{Theorem}

\newtheorem{lemma}[point]{Lemma}
\newtheorem{corollary}{Corollary}
\newcommand\blfootnote[1]{
    \begingroup
    \renewcommand\thefootnote{}\footnote{#1}
    \addtocounter{footnote}{-1}
    \endgroup
}

\begin{document}

\title[Hall Geometry and Auslander-Reiten Quiver]{Hall Geometry and\\ Auslander-Reiten Quiver}

\author[A Verma]{Aayush Verma}
\email{aayushverma6380@gmail.com}

\begin{abstract}
    We show how the geometric information in the motivic Hall algebra and the correspondence of the moduli stack recovers the Auslander-Reiten sequences and the Auslander-Reiten quiver.
\end{abstract} 

\maketitle
 \tableofcontents

 \vspace{-.6in}



\blfootnote{April 2026\\
    \leftskip=0cm
\rightskip=6cm
    \indent \it  ``Now, that is not wisely asked; has it not been told thee, that the gods made a bridge from earth, to heaven, called Bifröst?'' --- Gylfaginning, Prose Edda
}

\section{Introduction}
We will start with the definition of the Grothendieck ring of varieties $K_0(\text{Var}_k)$. Let $k$ be a field. We define a free abelian group generated by isomorphism classes $[X]$ where $X$ is a variety, with some relations, known as the cut and paste relation (or scissors relation), for every closed subvariety $Y$ of $X$, we write $[X] = [Y] + [X\backslash Y]$. The ring structure is given by $[X] \cdot [Y] \coloneq [X \times_k Y]$. 
This section will be rather pedagogical; the main result can be found in the next section.

We define the Lefschetz motive as $\L \coloneq [\ak^1]$ and define a localized ring $\M \coloneq \ko [\L^{-1}]$. The class $[X] \in \ko$ is called {\it motive/class} of variety $X \in \varki$. For example, the class $[\ak^n] = [\ak^1 \times \ak^2 \dots \times \ak^n] \cong \L^n$. Another example is to consider $\mathbb{P}^n$ which is 
\begin{align}
\mathbb{P}^n &= \a^n \sqcup \mathbb{P}^{n-1}\\
&=\a^n \sqcup \a^{n-1} \sqcup \dots \sqcup \a^0
\end{align}
so
\begin{equation}
[\mathbb{P}^n] = 1 + \L + \L^2 + \dots + \L^n
\end{equation}
which can be written as a geometric series  (when $(\L-1)$ is inverted)
\begin{equation}
[\mathbb{P}^n] = \frac{\L^{n+1}-1}{\L-1}
\end{equation}
Due to the result of \cite{borisov2014class}, we know that inverting $\L$ kills some non-zero classes, as $\L$ is a known zero-divisor in $\ko$ over fields of characteristic zero. Hence, there is a loss of information in the localized ring. In other words, the map $\ko \to \ko [\L^{-1}]$ is not injective for the field of characteristic zero.

We define a virtual motive for $[X]$ if $X$ is irreducible
\begin{equation}
[X]_{\vir} \coloneq (- \L^{-\dim X/2}) [X] \in \M
\end{equation} 
which is a dimension-dependent twist by $(- \L^{-\dim X/2})$. The square root of $\L$ and its multiplicative inverse are important in the theory of the Donaldson-Thomas invariants. For example, for $\P^1$, the virtual motive $[\P^1]_\vir = -(\L^{-1/2}+\L^{1/2})$. But here, technically, it is just a normalization because there is no singularity in $\P^1$. For another non-trivial virtual class, $[\P^2]_\vir = -(\L^{-1} + 1 +\L^{1})$. The first non-trivial and non-smooth example would be to consider $[\Hilb^2(\a^3)]_\vir$, see \cite{behrend2013motivic}, but we will not do it explicitly here.

Now, if $k=\mathbb{F}_q$, then a measure taking value in $\mathbb{Z}$ is just the counting measure $K_0(Var_k) \to \mathbb{Z}$, $[X] \mapsto \#X(\mathbb{F}_q)$. The Euler characteristic is also a measure (assume $k=\mathbb{C}$) $\chi_{\mathbb{C}}:K_0(var_{\mathbb{C}}) \to \mathbb{Z}$, $[X] \mapsto \chi_{\mathbb{C}}(X)$. And the Hodge-Deligne polynomial is also a measure encoding mixed Hodge numbers $K_0(Var_k) \to \mathbb{Z}[u,v]$, and this was the central theme of the birational Calabi-Yau discussion of string theory.

A {\it Quiver} $Q$ is a multi-directed graph with a set of vertices $Q_0$ and a set of arrows $Q_1$. For example, $A_3$ is
\[\begin{tikzcd}[ampersand replacement=\&]
    1 \& 2 \& 3
    \arrow["\alpha", from=1-1, to=1-2]
    \arrow["\beta", from=1-2, to=1-3]
\end{tikzcd}\]
and $K_2$ is
\[
\begin{tikzcd}
1
\arrow[r, "\alpha", shift left=1.5]
\arrow[r, "\beta"', shift right=1.5]
&
2
\end{tikzcd}
\]

We write a representation $V$ of $Q$ which consists of finite-dimensional vector spaces $V_i$ for each vertex $i \in Q_0$ and a linear map $T_\alpha: V_i \to V_j$ for $\alpha: i  \to j$ in $Q_1$. Two representations $((V_i),(T_\alpha))$ and $((V'_i),(T'_\alpha))$ are equivalent if there exists $\varphi_i$ and $\varphi_j$ such that the following diagram commutes for each $V_i$ and $T_\alpha$
\[\begin{tikzcd}
    {V_i} & {V_j} \\
    {V'_i} & {V'_j}
    \arrow["{T_\alpha}", from=1-1, to=1-2]
    \arrow["{\varphi_i}"', from=1-1, to=2-1]
    \arrow["{\varphi_j}", from=1-2, to=2-2]
    \arrow["{T'_\alpha}"', from=2-1, to=2-2]
\end{tikzcd}\]
We define the dimension vector $d \in \mathbb{N} Q_0$ of a quiver representation as a tuple $d = (\dim\ V_i)_{i \in Q_0}$. If we fix vector spaces $V_i$ on $Q$, then we write the representation space
\begin{equation}
R_d(Q) \coloneq \bigoplus_{\alpha: i \to j} Hom (V_i, V_j)
\end{equation}
so that the vector spaces $V_i$ are just points in an affine space. If $V_i \cong k^{d_i}$, then $R_d(Q) \cong \ak^{\sum_{\alpha:i \to j} d_i d_j}$.

\section{From Moduli Stacks to Auslander-Reiten Quiver}

In this section, we first recall the motivic Hall algebra and Hall correspondence\footnote{Or simply just correspondence of stacks}, identify its fibres, then recover the Auslander-Reiten sequences (henceforth, AR sequences) and AR quiver using AR theory and the coarse moduli space, culminating in Theorem~\ref{theorem1:label} for Dynkin quivers. The general results used in this work about Auslander-Reiten theory can be found in \cite{assem2006elements}.

Classification of representations has been a central theme in the study of quivers. For Dynkin quivers (those with finite representation type), it is given by Gabriel's theorem \cite{gabriel1972unzerlegbare}. But a more algebro-geometric approach for quiver representations is presented using the following. Fix a dimension vector $d$, and we define a reductive algebraic group
\begin{equation}
G_d  = \Pi_{i \in Q_0} GL(V_i)
\end{equation}
which acts on $R_d(Q)$ by a Q-graded analogue of conjugation
\begin{equation}
(g_i) \cdot (f_\alpha) \coloneq (g_j f_\alpha g^{-1}_i)_{\alpha:i \to j}
\end{equation}
and we will identify this as a base-change action. The orbits of this action will be the isomorphism class of representations of a fixed $d$ of the quiver $Q$.

The topology on the orbit space $R_d(Q)/G_d$ is not really helpful, as some non-isomorphic representations can lie in orbit-closure relations. Let us define a quotient (algebraic) stack $\M_d = \md$, which will be a moduli stack. But there is an interesting line of thought about motivic invariants that begins with the consideration of a formal power series. For a symmetric matrix $A$ defining a quiver, we write a partition function\footnote{This partition function formally takes value in a motivic quantum space.} or a formal power series, following \cite{reineke2024donaldson}
\begin{equation}
P_A(x) \coloneq \sum_{d \in \mathbb{N} Q_0} \frac{[R_d(Q_A)]_{\vir}}{[G_d]_\vir} x^d
\end{equation}
where $x^d = x_1^{d_1} \cdot \ldots \cdot x_n^{d_n}$ are monomials. Note that the moduli stack $\M_d$ and the coefficients $[R_d(Q_A)]_{\vir}$, $[G_d]_{\vir}$ are different objects. While the former is geometric, the latter is more enumerative; precisely, the coefficients are virtual motivic classes that we have defined earlier. Both $\M_d$ and $P_A(x)$ can be used to count the quiver representations, and they lead to some very interesting stories. The formal power series is a motivic shadow of counting quiver representations. Let us stick to the moduli stack $\M_d$ in this paper, as our motivation is to pass to the homological information.

Let us take a basic example of $A_2$ with linear orientation\footnote{Reineke's \cite{reineke2024donaldson} definition of $P_A(x)$ is for symmetric quivers, hence there is nothing for us to comment on that. Anyway, we will not discuss it in this note.}
\[\begin{tikzcd}
    1 & 2
    \arrow["\alpha", from=1-1, to=1-2]
\end{tikzcd}\]
and say $d=(1,1)$, then $R_d(Q) \cong \ak^1$ and $G_d = GL_1(k) \times GL_1(k) \cong \mathbb{G}_m \times \mathbb{G}_m$ which acts by $(g_1, g_2) \cdot (\lambda) = g_2 g_1^{-1}\lambda$. There are exactly two orbits for this example, the one corresponding to the split representation $S_1 \oplus S_2$ and the one corresponding to the indecomposable (projective) representation $P_1$. This is, in fact, true that $A_2$ has three indecomposables which are, as represented by dimension vectors, $S_1 = (1,0)$, $S_2 = (0,1)$, and $P_1 = (1,1)$. But the moduli stack $\M_d = \left[  \frac{\ak^1}{(\mathbb{G}_m \times \mathbb{G}_m)} \right]$ also remembers the stabilizers at points and records their automorphisms as different strata. Another reason for this moduli stack rather than a coarse quotient or orbit space will become clear in our next discussion, which is about the motivic Hall algebra.

We wish to understand how quiver representations connect with each other and what the structure of the category of representations $\rep$ (or equivalently, $\mod$, which is why in this paper, we will {\it often} interchange between each other) is. For this, we would like to study the extensions of representations, like $Ext^1_{\rep}(A,B)$
\begin{equation}\label{ses}
\xi_E: 0 \to B \to E \to A \to 0.
\end{equation}
We define $\M = \bigsqcup_{d \in \mathbb{N}Q_0} \M_d$ of all the finite dimensional representations of $Q$. (Quite formally, we define $\M^{(n)}$ as the moduli stack of n-flags of coherent sheaves on a base scheme \cite{Joyce:2005ta,bridgeland2012introduction}.) Let us take $\M^{(2)}$, which will be the stack of short exact sequences in $\rep$ in the Joyce-Bridgeland sense in our quiver setting. Higher $n$-flags also exist, but we are not interested in them here; for instance, $\M^{(3)}$ will be a moduli stack of $2-$steps filtrations.

Let us work with stacks that are locally of finite type. We define a correspondence of stacks, which can be called the Hall correspondence
{\label{stacks:label}
\[\begin{tikzcd}
    {\mathcal{M}^{(2)}} & {\mathcal{M}} \\
    {\mathcal{M} \times \mathcal{M}}
    \arrow["e", from=1-1, to=1-2]
    \arrow["{(b,a)}"', from=1-1, to=2-1]
\end{tikzcd}\]}
where the morphisms $a,b,e$ takes a short exact sequence \eqref{ses} to its constituent objects
\[\begin{tikzcd}
    0 & B & E & A & 0 \\
    {(B,A)} &&&& E
    \arrow[from=1-1, to=1-2]
    \arrow[from=1-2, to=1-3]
    \arrow[from=1-3, to=1-4]
    \arrow["{(b,a)}"{description}, from=1-3, to=2-1]
    \arrow["e"{description}, from=1-3, to=2-5]
    \arrow[from=1-4, to=1-5]
\end{tikzcd}\]
The fibre of $(b,a)$ over $(B,A) \in \M \times \M$ is the quotient moduli stack of extensions $[\Ext^1_{\rep} (A,B)/\Hom_{\rep}(A,B)]$ where $\Hom_{\rep}(A,B)$ remembers the automorphism data of objects in extensions \cite{schiffmann2006lectures}. This is true since fixing the endpoints of an extension, the automorphism group $\text{Aut}(\xi_E) \cong \Hom(A,B)$ where $\xi$ is an extension \eqref{ses}. The fibre\footnote{In the sheaf-theoretic viewpoint, the fibre of $b$ over $B \in \M$ is a Quot scheme.} of $e$ over $E$ is the moduli of sub-representations of $E$, which parametrises the short exact sequences with a fixed $E$. Some comments about the nature of morphisms $a,e,b$ can be found in \cite{bridgeland2012introduction}. The stack $[\Ext^1_{\rep} (A,B)/\Hom_{\rep}(A,B)]$ contains two kinds of information: the classes of extensions and the automorphisms that these extensions possess. In a coarse moduli space, we forget the automorphisms. Note that we have slowly progressed towards the homological side of $\rep$ from the geometric knowledge about the same. More will be supplied by the motivic Hall algebra for our study.

Recall the moduli stack $\M = \bigsqcup_{d\in \mathbb{N}Q_0}\md$.  We define the relative Grothendieck group $K_0(St/\M)$ to be the free abelian group generated by the morphisms of stacks $[X \xrightarrow{f} \M]$ where $X$ is an algebraic stack, modulo cut-and-paste relations
\begin{equation}
[X \xrightarrow{f} \M] = [Y \xrightarrow{f_{|Y}} \M] + [U \xrightarrow{f_{|U}} \M]
\end{equation}
for a closed substack $Y \subset X$ and $U=X\backslash Y$. We define {\it motivic Hall algebra} for $\rep$ to be
\begin{equation}
\hall \coloneq K_0(St/\M)
\end{equation}
where the product is defined by the correspondence of stacks $\M \times \M \leftarrow \M^{(2)} \to \M$ that we discussed earlier. Explicitly, we have
\begin{equation}
[X_1 \xrightarrow{f_1} \M]\ * [X_2 \xrightarrow{f_2} \M] = [Z \xrightarrow{e \circ h} \M]
\end{equation}
where $h$ is given the Cartesian square
\[\begin{tikzcd}
    Z && {\mathcal{M}^{(2)}} & {\mathcal{M}} \\
    {X_1 \times X_2} && {\mathcal{M} \times \mathcal{M}}
    \arrow["h", from=1-1, to=1-3]
    \arrow[from=1-1, to=2-1]
    \arrow["e", from=1-3, to=1-4]
    \arrow["{(b,a)}"', from=1-3, to=2-3]
    \arrow["{(f_1 \times f_2)}", from=2-1, to=2-3]
\end{tikzcd}\]
Hence, we claim that $X_1$ and $X_2$ are the families of subobjects and quotient representations in $\M$. We pullback $X_1 \times X_2$ to $\mathcal{M}^{(2)}$ along $(b,a)$ and then pushforward along $e$ to $\M$, and to a `first' approximation as in \cite{bridgeland2016hall}, that the Hall product of two such families is given by taking their universal extension. So, the motivic Hall algebra $\hall$ packages all the short exact sequences in $\rep$ or $\mod$ via the Hall correspondence. But there is room for a deeper statement that we will show below.

There are several Hall algebras formalisms: cohomological Hall algebras (CoHA, finitary Hall algebras, Hall algebras of constructible functions, but following Bridgeland \cite{bridgeland2016hall}, all of these Hall algebras can be thought of as different ways of taking ``(co)homology'' of the moduli stack of objects in an abelian category with convolution product provided by extension correspondence. In fact, our motivic Hall algebra is also a cohomology of a moduli stack in a sense, but we will not discuss it here. The motivation to do a Hall algebra varies through mathematics and physics; for instance, CoHA has been used in Reineke's program to understand Donaldson-Thomas invariants, stability conditions, and quantum torus algebra \cite{reineke2024donaldson}, and it has been used in physics to study BPS states as well \cite{kontsevich2011cohomological,kontsevich2008stability,harvey1998algebras,kucharski2017bps}. Anyway, we now move to our main result.

\lemma{\label{lemma1:label}
    The fibre $(b,a)$ over $(B,A): \M^2 \to \M \times \M$ is the moduli stack with the trivial action
    \begin{equation}
    [\Ext^1 (A,B) / \Hom(A,B)] \cong \Ext^1(A,B) \times B\Hom(A,B)
    \end{equation}
    where each extension $\xi \in \Ext^1(A,B)$ has the same stabilizer group $\Hom(A,B)$. The fibre is the coarse moduli space $\Ext^1(A,B)$ for which fibre dimension\footnote{For a quotient stack $[V/G]$, the geometric (or stacky dimension) is $\dim [V/G] = \dim V - \dim G$. Here we are only bothered about the coarse moduli space.} of the Hall correspondence is defined as the dimension of coarse extension space \\ $\epsilon_{A,B}:= \dim \Ext^1(A,B)$.
}

\proof{
    This is true since any automorphism $\phi$ of $0 \to B \to E \to A \to 0$ with fixed $A$ and $B$ has the form $\phi = id_E + \iota f \pi $ for every $f \in \Hom(A,B)$. Indeed, one can show that $\Hom(A,B) \to \Aut(\xi_E)$ ($f \mapsto id_E + \iota f \pi$) is an isomorphism. An automorphism $\phi$ does not change the Yoneda class of extension $\xi_E$. Hence, the action of $\Hom(A,B)$ is trivial. Every point in $\Ext^1(A,B)$ has the same stabilizer group $\Hom(A,B)$. 

 Because the quotient is by a trivial action, we can write
 \begin{equation}
 [\Ext^1 (A,B) / \Hom(A,B)] \cong \Ext^1(A,B) \times B\Hom(A,B)
 \end{equation}
The coarse moduli space (forgetting the stabilizers) is $\Ext^1(A,B)$.
}

\remark{
    For our hereditary algebra, $kQ$, the stack dimension of the moduli stack $[\Ext^1(A,B)/ \Hom(A,B)]$ is $\dim \Ext^1(A,B)-\dim \Hom(A,B) = - \langle \dim A, \dim B \rangle_Q$ where $\langle \dim A, \dim B \rangle_Q$ is the Euler form which is a combinatorial information.
}

\lemma[AR Duality for Hereditary Algebra]{\label{lemma2:label}
    AR duality is a natural isomorphism $D\overline{\Hom}_k(B,\tau A) \cong \Ext^1(A, B)$ where $D=Hom_k(-,k)$. For hereditary algebra (gl.dim$ \leq 1$), there is no stabilization needed for Hom-space, so one simply writes (after dualizing both sides)
    \begin{equation}\label{ARduality}
     D\Ext^1(A,B) \cong \Hom_k(B, \tau A)
     \end{equation} 
 for all $A$ and $B$ in $\mod$.

\proof{
    See Chapter~4 (Corollary~2.14) in \cite{assem2006elements}.
}

\remark{\label{remark:label}
    In Dynkin case, when we set $B= \tau A$, then $\End_k(\tau A) \cong k$ so $DExt^1(A, \tau A) \cong k$ and dualizing both sides give $Ext^1(A,\tau A) \cong k$. Hence, the Hall fibre over $(\tau A,A)$ has the {\it coarse} moduli space $\mathbb{A}^1$.

    However, the fact that $\Ext^1(A,\tau A) \cong k$ is not necessarily true for tame or wild algebras.
}

\theorem{\label{theorem1:label}
    Let $k$ be an algebraically closed field and $\mod$ be the path algebra for a Dynkin quiver $Q$. The Ext-dimension functions $\{ \ext \}_{A  \text{ non-proj} \in \mod}$, from the Hall correspondence, can determine the Auslander-Reiten sequences $0 \to \tau A \to E \to A \to 0$. 
}

\proof{
    By Lemma~\ref{lemma2:label}, we write $\ext = \dim \Hom (B, \tau A)$ where $A,B \in \mod$ are indecomposables and $A$ is a non-projective module. Since $\mod$ is Hom-finite, $\dim \Hom (B, \tau A)$ for all the indecomposables $B$ determine the representable functor $\Hom(-,\tau A)$. Hence, by the Yoneda lemma, together with the Krull-Schmidt property, this uniquely determines $\tau A$.

    From Lemma~\ref{lemma1:label} and Remark~\ref{remark:label}, the coarse moduli space of fibre over $(\tau A, A)$ is $\Ext^1(A, \tau A) \cong \mathbb{A}^1$. Hence, the stratification is just 
    \begin{equation}
    \A^1 = \{0\} \sqcup k^\times
    \end{equation}
    where $\{0\}$ corresponds to the split extension\footnote{Recall that $\Ext^1(A,B)$ as an abelian group with Baer sum addition has an identity\\ $0 \to B \to B \oplus A \to A \to 0$.}
    \begin{equation}
    0 \to \tau A \to \tau A \oplus A \to A \to 0
    \end{equation}
    and $k^\times$ corresponds to the non-split extension classes in $\Ext^1(A, \tau A)$.  Moreover, the Auslander-Reiten theorem characterizes the class of AR sequences $\delta$  generating the socle of $\Ext^1(A, \tau A)$ as a left $\End(\tau A)$-module. Since $\End(\tau A) \cong k$, one has $\Ext^1(A, \tau A)$, which is simple, equal to its socle. Thus, we get that $\delta \in k^\times $, which is the non-split stratum of our coarse moduli space. This provides a way of obtaining AR theory from the correspondence of stacks $\M \times \M \leftarrow \M^{(2)} \to \M$.
}

\corollary{
    For a Dynkin quiver $Q$, we can obtain the full mesh of Auslander-Reiten quiver $\Gamma(\mod)$ from the Ext-dimension function $\{ \ext \}_{A \text{ non-proj}}$.
}

\proof{
    Any two nonzero classes $\delta, \delta' \in k^\times$ satisfy $\delta' = \lambda \delta$, which makes $(id_{\tau A}, \lambda id_{E}, id_{A})$ an isomorphism of extensions, so the $E$ is determined up to isomorphism by the open stratum. We know that $\mod$ is a k-linear Hom-finite with $\End(X) \cong k$ for all indecomposables $X$. Hence, Krull-Schmidt gives a unique decomposition $E \cong \bigoplus_X  X^{a_X}$ as indecomposable summands which are the sources of the irreducible maps into $A$ in the mesh $\Gamma (\mod)$. Hence, applying the previous theorem for all the indecomposables (non-projectives), we can reconstruct $\Gamma(\mod)$ completely.
}


Note that the stratification data in coarse moduli space $\Ext^1(A,\tau A)$ contains both split and non-split representations corresponding to the zero and non-zero point in the geometry, but the coarse moduli space distinguish between the split and non-split in the data of extensions, which we need for AR sequences, but AR theory further organises these non-split extensions to almost split extensions \cite{schiffler2014quiver}. But since our case is Dynkin, the almost split sequence organizes itself in the open stratum of $\A^1$. 

Our point in this note was that quiver moduli stacks contain more than just counting. At Hall correspondence, it already knows about the structure of extensions on $\rep$. Once we recognize it, we can pass to AR theory. Thus, the passage from quiver moduli to the motivic Hall algebra and Hall geometry, and then to AR theory, can be thought of as a journey from enumerative geometry to homological representation theory. We chose AR theory as a specific interest; we can generalize this bridge to other homological interests. 

{\it Acknowledgments :} I am indebted to Amit Kuber for introducing and teaching me about quivers and their representations.

\appendix
\section{AR Triangles and Happel's Theorem for Triangulated Categories}

This appendix is presented for intellectual consumption about AR triangles in $D^b(\mod)$.

There have been attempts to understand the derived Hall algebra \cite{toen2006derived,xiao2008hall,kontsevich2008stability}. But our motivation and purpose differ from those of derived Hall algebras. The passage from geometric information to our correspondence of stacks leads to the structure of the abelian category $\rep$ and thus the Auslander-Reiten theory as discussed above. Similar to Auslander-Reiten (AR) sequences in $\mod$, we study the Auslander-Reiten triangles in $D^b(\mod)$. AR sequences in $\mod$ are {almost split sequences} $Ext^1(A,\tau A)$
\begin{equation}
0 \to \tau A \to E \to A \to 0
\end{equation}
where $\tau$ is the AR translate, $A$ and $\tau A$ are indecomposables, and $A$ is a non-projective module. For $A_2$, we have an AR sequence
\begin{equation}
0 \to S_2 \to P_1 \to S_1 \to 0
\end{equation}
In $D^b(\mod)$, an AR triangle is
\begin{equation}
\tau A \to E \to A \to \tau A[1]
\end{equation}
where $[1]$ denotes the shift (an automorphism of the category). But studying such triangles is non-trivial. For this, we find a triangulated category which is equivalent to $D^b(\mod)$ and study the AR triangles in that category. We know that $\mod$ is a hereditary algebra and hence has finite global dimension (in fact, one), but in general, for any finite-dimensional algebra of finite global dimension, we have the following theorem due to Happel \cite{happel1988triangulated}, also see \cite{happel1991auslander}.
\begin{theorem}[Happel]
Let $\la$ be a finite dimensional $k-$algebra and $\la$ has finite global dimension ($\text{gl.dim} < \infty$). There exists a triangle equivalence between categories defined by a functor
\begin{equation}
\mu: D^b(\text{mod-}\Lambda) \to \underline{\text{mod-}}\hat{\Lambda}
\end{equation}
where $\underline{\text{mod-}}\hat{\la}$ is the stable module category over repetitive algebra $\hat{\la}$.
\end{theorem}

We presented in a talk\footnote{“Some Aspects of Happel’s Theorem and the Auslander-Reiten Triangles,” 2025.
Slides available at https://aayushayh.github.io/talks/mth619presentation.pdf.} how AR triangles are easy to understand in a stable module category $\underline{\text{mod-}}\hat{\la}$ which is a triangulated category\footnote{Happel shows \cite{happel1988triangulated} that $\underline{\text{mod-}}\hat{\la}$ is a Frobenius category and a stable category of a Frobenius category is a triangulated category.} and under this functor,one can study AR triangles in $D^b(\text{mod-}\l)$. Because it is easier to understand the shift functor (which happens to be the cosyzygy functor $\Omega^{-1}$) in $\underline{\text{mod-}}\hat{\la}$.

It would be interesting to see whether there is a more direct route to studying AR triangles in geometric terms, similar to our correspondence between the Hall geometry and the AR quiver for an abelian category.

\bibliography{motivic}
\bibliographystyle{utphys}

\end{document}